\documentclass[]{article}
\usepackage[dvipsnames]{xcolor}

\usepackage[numbers,sort&compress]{natbib}% Citation support using natbib.sty

\usepackage[lmargin=3.5cm, rmargin=3.5cm]{geometry}

\usepackage[utf8]{inputenc}
\usepackage{lmodern}
\usepackage{mathtools} \usepackage{amssymb}

\DeclareMathOperator{\cl}{cl}
\DeclareMathOperator{\conv}{conv}
\DeclareMathOperator{\cone}{cone}

\DeclareMathOperator{\interior}{int}
\DeclareMathOperator{\vertices}{vert}
\DeclareMathOperator{\ext}{ext}
\DeclareMathOperator{\argmin}{argmin}
\DeclareMathOperator{\argmax}{argmax}

\usepackage{amsthm}
\newtheorem{theorem}{Theorem}[section]
\newtheorem{proposition}[theorem]{Proposition}
\newtheorem{lemma}[theorem]{Lemma}
\newtheorem{corollary}[theorem]{Corollary}
\theoremstyle{definition}
\newtheorem{definition}[theorem]{Definition}
\newtheorem{remark}[theorem]{Remark}
\newtheorem{example}{Example}[section]

\newcommand{\recc}[1]{0^+ {#1}} %recession cone
\newcommand{\polar}[1]{{#1}^{\circ}} %polar cone
\newcommand{\R}{\mathbb{R}}
\newcommand{\T}{^{\mathsf{T}}}
\newcommand{\set}[1]{\left\lbrace {#1} \right\rbrace}
\newcommand{\norm}[1]{\left\lVert {#1} \right\rVert}
\newcommand{\haus}[2]{d_{\mathsf{H}}\left({#1},{#2}\right)} %Hausdorff distance
\newcommand{\thaus}[2]{\bar{d}_{\mathsf{H}}(#1,#2)} %truncated
\newcommand{\exc}[2]{e[{#1},{#2}]} % excess of #1 over #2
\newcommand{\ed}{\left(\varepsilon , \delta\right)}
\newcommand{\mgeq}{\succcurlyeq} % curly \geq / "matrig geq"
\newcommand{\mg}{\succ} % curly > / "matrix greater"
\newcommand{\A}{\mathcal{A}}
\newcommand{\Ab}{\bar{\A}}
\newcommand{\decoRule}{\rule{.8\textwidth}{.4pt}}

\newcommand{\ols}[1]{\mskip.5\thinmuskip\overline{\mskip-.5\thinmuskip
    {#1} \mskip-.5\thinmuskip}\mskip.5\thinmuskip} % overline that
\usepackage{hyperref}
\usepackage{color}
\usepackage{enumerate}
\usepackage{subcaption}
\usepackage{booktabs}
\usepackage{makecell}
\usepackage{diagbox}
\usepackage[misc]{ifsym}
\usepackage[linesnumbered,algoruled,vlined]{algorithm2e}
\usepackage{ifthen}
\newcommand{\isassigned}{\leftarrow}

\begin{document}

\title{On the Approximation of Unbounded Convex Sets by Polyhedra}

\author{Daniel Dörfler\thanks{daniel.doerfler@uni-jena.de}\\Friedrich Schiller University Jena, Germany}

\maketitle

\begin{abstract}
  This article is concerned with the approximation of unbounded convex
  sets by polyhedra. While there is an abundance of literature
  investigating this task for compact sets, results on the unbounded
  case are scarce. We first point out the connections
  between existing results before introducing a new notion of
  polyhedral approximation called $\ed$-approximation that integrates
  the unbounded case in a meaningful way. Some basic results about
  $\ed$-approximations are proven for general convex sets. In the last
  section an algorithm for the computation of $\ed$-approximations of
  spectrahedra is presented. Correctness and finiteness of the
  algorithm are proven.
\end{abstract}

\section{Introduction}
The problem of approximating a convex set $C$ by a polyhedron in the
Hausdorff distance has been studied systematically for at least a
century \cite{Min03} and has a variety of applications in mathematical
programming. These include algorithms for convex optimization problems
that approximate the feasible region by a sequence of polyhedra
\cite{CG59, Kel60} or solution concepts for convex vector optimization
problems \cite{RW05, ESS11, LRU14}. Moreover, there are multiple
algorithms for mixed-integer convex optimzation problems that are
based on polyhedral outer approximations, see \cite{DG86, WP95,
KLW16}. Interest in this problem is fueled by the fact that polyhedra
have a simple structure in the sense that they can be described by
finitely many points and directions. This finite structure makes
computations with polyhedra more viable than with general convex
sets. Hence, it is desirable to work with polyhedra that approximate
some complicated set well. If the set $C$ to be approximated is
assumed to be compact, then numerous theoretical results are
available. This includes asymptotic \cite{FT48, Sch81} and explicit
\cite{BI75} bounds on the number of vertices that a polyhedron needs
to have in order to approximate $C$ to within a prescribed
accuracy. Moreover, iterative algorithms, so called augmenting and
cutting schemes, for the polyhedral approximation of convex bodies are
known, see \cite{Kam92}. Some convergence properties are discussed in
\cite{Kam93} and \cite{Kam94}. An overview about combinatorial aspects
of the approximation of convex bodies is collected in the survey
article by Bronshte\u{\i}n \cite{Bro07}.

If boundedness of the set $C$ is not assumed, the amount of literature
on the problem is scarce, although there are various applications
where unbounded convex sets arise naturally. In convex vector
optimization, for example, it is known that the so-called extended
feasible image or upper image contains the set of nondominated points
on its boundary, see e.g. \cite{Loe11, ESS11}. Due to its geometric
properties it is advantageous to work with this unbounded set instead
of with the feasible image itself. Another application is in large
deviations theory, which, generally speaking, is the study of the
asymptotic behaviour of tails of sequences of probability
distributions, see \cite{Var84}. Under certain conditions bounds for
this behaviour can be obtained in terms of rate functions of
probability distributions. In \cite{NR95} such bounds are obtained
under the condition that the level sets of a specific convex rate
function can be approximated by polyhedra. Moreover, the authors of
\cite{LYBV18} generalize the algorithm in \cite{DG86} and consider
mixed-integer convex optimization problems, whose feasible region is
not necessarily bounded. The problems are solved by computing
polyhedral outer approximations of the feasible region in such a
fashion that reaching a globally optimal solution is guaranteed.

The most notable result about the polyhedral approximation of $C$ is
due to Ney and Robinson \cite{NR95} who give a characterization of the
class of sets that can be approximated arbitrarily well by polyhedra
in the Hausdorff distance. However, this class is relatively small as
restrictive assumptions on the recession cone of $C$ have to be made,
such as polyhedrality. The reason boils down to the fact that the
Hausdorff distance is seldom suitable to measure the similarity
between unbounded sets. In fact, the Hausdorff distance between closed
and convex sets is finite only if the recession cones of the sets are
equal, see Proposition \ref{PHDRC} in Section \ref{SEC_3}. Due to this
difficulty, additional assumptions about the structure of the problem
have to be made in each of the aforementioned applications. These
include polyhedrality of the ordering cone or boundedness of the
problem in convex vector optimization, see e.g. \cite{LRU14, DLSW21},
polyhedrality of a cone generated by the rate function in large
deviations theory, or strong duality when dealing with convex
optimization problems. In 2018, Ulus \cite{Ulu18} characterized the
tractability of convex vector optimization problems in terms of
polyhedral approximations. One important necessary condition is the
so-called self-boundedness of the problem.

Considering the facts mentioned, polyhedral approximation of unbounded
convex sets requires a notion that does not solely rely on the
Hausdorff distance. To this end our main contribution is the
introduction of the notion of $\ed$-approximation for closed convex
sets $C$ that do not contain lines. One feature of
$\ed$-approximations is that the recession cones of the involved sets
play an important role. We show that $\ed$-approximations define a
meaningful notion of polyhedral approximation in the sense that a
sequence of approximations converges to the set $C$ as $\varepsilon$
and $\delta$ diminish. This convergence is achieved in the sense of
Painlevé-Kuratowski, who define convergence for sequences of sets, see
\cite{RW98}. Moreover, we present an algorithm for the computation of
$\ed$-approximations when the set $C$ is a spectrahedron, i.e. defined
by a linear matrix inequality. We also prove correctness and
finiteness of the algorithm. Its main purpose, however, is to show
that $\ed$-approximations can be constructed in finitely many steps
theoretically.

This article is organized as follows. In the next section we introduce
the necessary notation and provide definitions. In Section \ref{SEC_3}
we compare the results by Ney and Robinson \cite{NR95} with the
results by Ulus \cite{Ulu18} and put them in relation. In particular,
we show that self-boundedness is a special case of the property that
the excess of a set over its own recession cone is finite. The concept
of $\ed$-approximations is introduced in Section \ref{SEC_4}. We prove
a bound on the Hausdorff distance between truncations of an
$\ed$-approximation and truncations of $C$. The main result is Theorem
\ref{TPKC}. It states that a sequence of $\ed$-approximations of $C$
converges to $C$ in the sense of Painlevé-Kuratowski as $\varepsilon$
and $\delta$ tend to zero. In the last section we present the
aforementioned algorithm and prove correctness and finiteness as well
as illustrate it with two examples.

\section{Preliminaries}\label{SEC_2}
Throughout this article we denote by $\cl C$, $\interior C$, $\recc
C$, $\conv C$, and $\cone C$ the closure, interior, recession cone,
convex hull and conical hull of a set $C$, respectively. A compact
convex set with nonempty interior is called a \textit{convex
body}. The Euclidean ball with radius $r$ centered at a point~${c \in
\R^n}$ is denoted by~$B_r(c)$. A point $c$ of a convex set $C$ is
called an \emph{extreme point} of $C$, if $C \setminus\{c\}$ is
convex. The set of extreme points of $C$ is denoted by $\ext
C$. Extreme points are exactly the points of $C$ that cannot be
written as a proper convex combination of elements of $C$
\cite[p.~162]{Roc70}. For finite sets $V, D \subseteq \R^n$, the set
\begin{equation}\label{vrep}
P = \conv V + \cone D
\end{equation}
is called a \emph{polyhedron}. The plus sign denotes Minkowski
addition. The sets $V,D$ in \eqref{vrep} are called a
\emph{$V$-representation} of $P$ as it is expressed in terms of its
vertices and directions. A polyhedron can equivalently be expressed as
a finite intersection of closed halfspaces \cite[Theorem 19.1]{Roc70},
i.e.
\begin{equation}\label{hrep}
P = \{x \in \R^n \mid Ax \leq b\}
\end{equation}
for a matrix $A \in \R^{m \times n}$ and a vector $b \in \R^m$. The
data $(A,b)$ are called a \emph{$H$-representation} of $P$.  The
extreme points of a polyhedron $P$ are called \emph{vertices} of $P$
and are denoted by~$\vertices P$. For symmetric matrices
$A_0,A_1,\dots,A_n$ of arbitrary fixed size we define
\begin{equation}
\Ab(x) := \sum_{i=1}^n x_iA_i \quad \text{and} \quad \A(x) := \Ab(x) +
A_0,
\end{equation}
i.e. a linear combination of the $A_i$ and a translation of $\Ab(x)$
by $A_0$, respectively. We denote by $A \mg 0$ ($A \mgeq 0$) positive
(semi-)definiteness of the symmetric matrix $A$. A set of the form
$\{x \in \R^n \mid \A(x) \mgeq 0\}$ is called a
\emph{spectrahedron}. Spectrahedra are a generalization of polyhedra
for which many geometric properties of polyhedra generalize nicely,
e.g. the recession cone of a spectrahedron $C$ is obtained as $\{x \in
\R^n \mid \Ab(x) \mgeq 0 \}$, see \cite{GR95}, whereas the recession
cone of a polyhedron in $H$-representation is $\{x \in \R^n \mid Ax
\leq 0\}$. Given a cone $K$ the set $\polar{K} = \{y \in \R^n \mid
\forall x \in K: x\T y \leq 0\}$ is called the \emph{polar cone} of
$K$. The polar $\polar{(\recc{C})}$ of the recession cone of a
spectrahedron $C$ is computed as
\begin{equation}\label{EPCS}
\polar{(\recc{C})} = \cl \left\lbrace \left( -A_1 \cdot X, \dots, -A_n
\cdot X \right)\T \mid X \mgeq 0 \right\rbrace, 
\end{equation}
where $A_i \cdot X$ means the trace of the matrix product $A_iX$, see
\cite[Section 3]{GR95}. A cone $K$ is called \emph{polyhedral} if ${K
= \cone D}$ for some finite set $D$ and \emph{pointed} if ${K \cap
(-K) = \{0\}}$. A set whose recession cone is pointed is called
\emph{line-free}. It is well known, that a closed convex set contains
an extreme point if and only if it is line-free, see \cite[Corollary
18.5.3]{Roc70}. Given nonempty sets ${C_1, C_2 \subseteq \R^n}$ the
\emph{excess of $C_1$ over $C_2$}, ${\exc{C_1}{C_2}}$, is defined as
\begin{equation}
\exc{C_1}{C_2} = \adjustlimits \sup_{c_1 \in C_1} \inf_{c_2 \in C_2}
\norm{c_1-c_2},
\end{equation}
where~$\norm{\cdot}$ denotes the Euclidean norm. The \emph{Hausdorff
distance between $C_1$ and $C_2$}, ${\haus{C_1}{C_2}}$, is then
expressed as
\begin{equation}
\haus{C_1}{C_2} = \max \{\exc{C_1}{C_2}, \exc{C_2}{C_1}\}.
\end{equation}
It is well known, that the Hausdorff distance defines a metric on the
space of nonempty compact subsets of $\R^n$. Between unbounded sets
the Hausdorff distance may be infinite. A polyhedron $P$ is called an
\textit{$\varepsilon$-approximation} of a convex set $C$ if
$\haus{P}{C} \leq \varepsilon$.

\section{Polyhedral Approximation in the Hausdorff Distance}\label{SEC_3}
Every convex body can be approximated arbitrarily well by a polytope
in the Hausdorff distance, see e.g. \cite{Sch81}. Moreover, algorithms
for the computation of $\varepsilon$-approximations exist for which
the convergence rate is known \cite{Kam92}. For the approximation of
unbounded convex sets only the following theorem is known, which
provides a characterization of the sets that can be approximated by
polyhedra in the Hausdorff distance.

\begin{theorem}[see {\cite[Theorem 2.1]{NR95}}]\label{TNR}
Let $C \subseteq \R^n$ be nonempty and closed. Then the following are
equivalent:
\begin{enumerate}[(i)]
\item $C$ is convex, $\recc{C}$ is polyhedral, and
  $\exc{C}{\recc{C}} < +\infty$.
\item \label{TNR_ii} There exists a polyhedral cone $K$ such that for every
  $\varepsilon > 0$ there exists a finite set $V \subseteq \R^n$
  such that~${\haus{\conv V + K}{C} \leq \varepsilon}$.
\end{enumerate}
Further, if (\ref{TNR_ii}) holds then $K = \recc{C}$.
\end{theorem}

A related result is found in \cite{Ulu18}, where the approximate
solvability of convex vector optimization problems in terms of
polyhedral approximations is investigated. In order to state the
result and establish the relationship to Theorem \ref{TNR} we need the
following definition from \cite{Ulu18}.

\begin{definition}
  A set $C \subsetneq \R^n$,  with a nontrivial recession cone is called
  \emph{self-bounded}, if there exists $y \in \R^n$ such
  that~${\{y\}+\recc{C} \supseteq C}$.
\end{definition}

Adjusted to our notation, the mentioned result can
be stated as:

\begin{proposition}[{see \citep[Proposition 3.7]{Ulu18}}]
  Let $C \subseteq \R^n$ be closed and convex. If $C$ is self-bounded,
  then there exists a finite set $V \subseteq \R^n$ such
  that~${\haus{\conv V + \recc{C}}{C} \leq \varepsilon}$.
\end{proposition}

If $\recc{C}$ is polyhedral, then, clearly, $C$ can be approximated by
a polyhedron. The difference to Theorem \ref{TNR} is the
self-boundedness of $C$ instead of the finiteness
of~${\exc{C}{\recc{C}}}$. The following theorem points out the
connection between these conditions and shows, that under an
additional assumption, both coincide. The relationships are
illustrated in Figure \ref{fig_exc_slfbnd}.

\begin{theorem}\label{TESB}
Given a nonempty, closed and convex set $C \subseteq \R^n$, consider
the statements
\begin{enumerate}[(i)]
\item\label{TESB_i} $\exc{C}{\recc{C}} < +\infty$,
\item\label{TESB_ii} $C$ is self-bounded,
\item\label{TESB_iii} There is a compact set $K \subseteq \R^n$ such
  that~${K + \recc{C} \supseteq C}$.
\end{enumerate}
Then the following implications are true: (\ref{TESB_i})
$\Leftrightarrow$ (\ref{TESB_iii}) and (\ref{TESB_ii}) $\Rightarrow$
(\ref{TESB_iii}). If, additionally, $\recc{C}$ is solid, then
(\ref{TESB_i}) -- (\ref{TESB_iii}) are equivalent.
\end{theorem}

\begin{proof}
  We start with the assertion \textit{(\ref{TESB_i})} $\Rightarrow$
  \textit{(\ref{TESB_iii})}. Let~${M = \exc{C}{\recc{C}}}$,~${c
    \in C}$ be arbitrary and~${r_c \in \recc{C}}$ such
  that~${\norm{c-r_c} = \inf_{r \in \recc{C}} \norm{c-r}}$. This
  infimum is uniquely attained, because~$\recc{C}$ is closed and
  convex. Then~${\norm{c-r_c} \leq M}$ and we conclude~${c =
    (c-r_c)+r_c \in B_M(0) + \recc{C}}$. Therefore,~${B_M(0) + \recc{C}
    \supseteq C}$.

  To show \textit{(\ref{TESB_iii})} $\Rightarrow$
  \textit{(\ref{TESB_i})}, let~${K + \recc{C} \supseteq C}$ for some
  compact set~$K$. Then we have
  \begin{equation*}
  \begin{aligned}
    \exc{C}{\recc{C}} & \leq \exc{K+\recc{C}}{\recc{C}} \\
    & = \adjustlimits \sup_{\substack{k \in K\\ r \in \recc{C}}}
    \inf_{\bar{r}
      \in \recc{C}} \norm{(k+r)-\bar{r}} \\
    & \leq \sup_{\substack{k \in K\\ r \in \recc{C}}} \norm{(k+r)-r}
    \\
    & = \sup_{k \in K} \norm{k} < +\infty.
  \end{aligned}
  \end{equation*}
  The implication \textit{(\ref{TESB_ii})} $\Rightarrow$
  \textit{(\ref{TESB_iii})} is trivial with~${K = \{y\}}$. For the last
  part we show \textit{(\ref{TESB_iii})} $\Rightarrow$
  \textit{(\ref{TESB_ii})}, assuming $\recc{C}$ is
  solid. If~${\exc{C}{\recc{C}} = 0}$, then~${C \subseteq \recc{C}}$ and
  we can set~${y = 0}$. Now, let~${\exc{C}{\recc{C}} = M > 0}$ and
  fix~${c \in \interior C}$. Then there exists~${\varepsilon > 0}$ such
  that~${B_{\varepsilon}(c) \subseteq \recc{C}}$. We have
  \begin{equation*}
    \begin{aligned}
      B_M\left(\frac{M}{\varepsilon} c\right) & =
      \frac{M}{\varepsilon} B_{\varepsilon}(c) \\
      & \subseteq \frac{M}{\varepsilon} \recc{C} \\
      & = \recc{C}.
    \end{aligned}
  \end{equation*}
  Therefore,~${\left\lbrace -\frac{M}{\varepsilon} c \right\rbrace +
    \recc{C} \supseteq B_M(0)}$ and, since~$\recc{C}$ is
  convex, $\left\lbrace -\frac{M}{\varepsilon} c \right\rbrace +
    \recc{C} \supseteq B_M(0) + \recc{C}$. From the first part of the
  proof we know, that~${B_M(0) + \recc{C} \supseteq C}$, which completes
  the proof.
\end{proof}

\begin{figure}
  \centering
  \includegraphics[scale=.9]{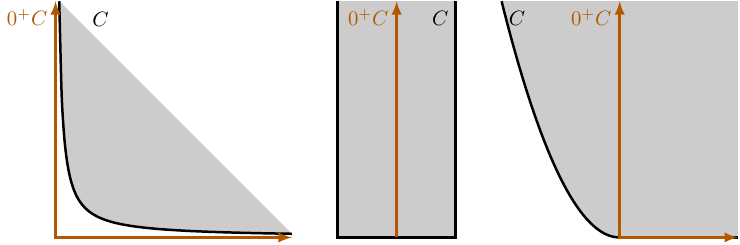}
  \decoRule
  \caption{\label{fig_exc_slfbnd} Illustration of Theorem
    \ref{TESB}. \textbf{Left:} The set $C$ is contained in its own
    recession cone. Therefore, it is self-bounded and
    $\exc{C}{\recc{C}}=0$. \textbf{Center:} The excess of $C$ over its
    recession cone is finite and attained in any of the two
    vertices. However, $C$ is not self-bounded, because it cannot be
    contained in a translation of $\recc{C}$. \textbf{Right:} A set that
    is neither self-bounded nor does it hold
    $\exc{C}{\recc{C}}<\infty$. Traversing the parabolic arc, the
    distance to $\recc{C}$ grows without bound.}
\end{figure}

\begin{example}
  To see that ${\exc{C}{\recc{C}} < +\infty}$ does not imply
  self-boundedness of $C$ unless $\recc{C}$ is solid, consider the
  following counterexample. In $\R^2$ let ${C = \conv \{\pm e_1\} +
    \cone \{e_2\}}$, where $e_i$ denotes the $i$-th unit
  vector. Then one has the equality $\exc{C}{\recc{C}}=1$, but $C$ is
  not self-bounded. The set is illustrated in the center of Figure
  \ref{fig_exc_slfbnd}.
\end{example}

In view of the above result, we suggest calling a set self-bounded, if it
satisfies Property~\textit{(\ref{TESB_iii})}. On one hand, this extends
the notion to sets whose recession cone is not solid. And on the other
hand, makes every compact set self-bounded, rather than just
singletons. Since in \cite{Ulu18} cones are assumed to be solid,
Theorem \ref{TESB} proves that a convex vector optimization problem is
tractable in terms of polyhedral approximations, if and only if the
upper image \cite[Equation 6]{Ulu18} of the problem satisfies
\textit{(i)} in Theorem \ref{TNR}.

The reason that many unbounded convex sets are beyond the scope of
polyhedral approximation in the Hausdorff distance is that it is by
nature designed to behave nicely only for compact sets. The following
proposition specifies this.

\begin{proposition}\label{PHDRC}
 For closed and convex sets $C_1, C_2 \subseteq \R^n$ it is true,
 that $\haus{C_1}{C_2} < +\infty$ only if~${\recc{C_1} =
   \recc{C_2}}$.
\end{proposition}

\begin{proof}
  Assume~${\recc{C_1} \neq \recc{C_2}}$ and let w.l.o.g.~${r \in
    \recc{C_1} \setminus \recc{C_2}}$. Consider the equivalent definition
  of the Hausdorff distance:
  \[ \haus{C_1}{C_2} = \inf \left\lbrace \varepsilon > 0 \mid C_1
      \subseteq C_2 + B_{\varepsilon}(0), C_2 \subseteq C_1 +
      B_{\varepsilon}(0) \right\rbrace. \] Let $\varepsilon$ be large enough
  such that~${C_1 \cap \left(C_2 + B_{\varepsilon}(0)\right) \neq
    \emptyset}$ and let $z$ be an element of this set. Then~${z +\mu r \in
    C_1}$ for all~${\mu \geq 0}$. The recession cone of~${C_2+B_e(0)}$
  is~$\recc{C_2}$ according to \cite[Proposition
  9.1.2]{Roc70}. Therefore, there exists some~$\mu_{\varepsilon}$ such
  that~${z+\mu r \notin C_2+B_e(0)}$ for all~${\mu \geq
    \mu_{\varepsilon}}$. This yields~${\haus{C_1}{C_2} \geq \varepsilon}$
  and the claim follows with~${\varepsilon \to \infty}$.
\end{proof}

\section{A Polyhedral Approximation Scheme for Closed Convex Line-Free
  Sets}\label{SEC_4}
We have seen that, in order to approximate a set $C$ by a polyhedron
$P$ in the Hausdorff distance, their recession cones need to be
identical. Theorems \ref{TNR} and \ref{TESB} tell us that this is
achievable only for specific sets $C$. To treat a larger class of
sets, a concept is needed that quantifies similarity between closed
convex cones, similar to how the Hausdorff distance quantifies
similarity between compact sets.

\begin{definition}
  Given nonempty closed convex cones~${K_1, K_2 \subseteq \R^n}$,
  the \emph{truncated Hausdorff distance between~$K_1$
    and~$K_2$},~${\thaus{K_1}{K_2}}$, is defined as
  \begin{equation}
    \thaus{K_1}{K_2} := \haus{K_1 \cap B_1(0)}{K_2 \cap B_1(0)}.
  \end{equation}
\end{definition}

Since every cone contains the origin, it is immediate
that~${\thaus{K_1}{K_2} \leq 1}$. The truncated Hausdorff distance
defines a metric on the set of closed convex cones in~$\R^n$, see
\cite{Gur65}. However, it is only one way among many to measure the
distance between convex cones. We suggest the survey in \cite{IS10}
for a more thorough discussion of the topic. With the truncated
Hausdorff distance we define the following notion of polyhedral
approximation of convex sets that are not necessarily bounded.

\begin{definition}\label{DEDA}
  Given a nonempty closed convex and line-free set~${C \subseteq
    \R^n}$, a line-free polyhedron~$P$ is called an
  \emph{${\ed}$-approximation of~$C$} if
  \begin{enumerate}[(i)]
  \item $\exc{\vertices P}{C} \leq \varepsilon$,
  \item $\thaus{\recc{P}}{\recc{C}} \leq \delta$,
  \item $P \supseteq C$.
  \end{enumerate}
\end{definition}

\begin{figure}
  \centering
  \includegraphics[]{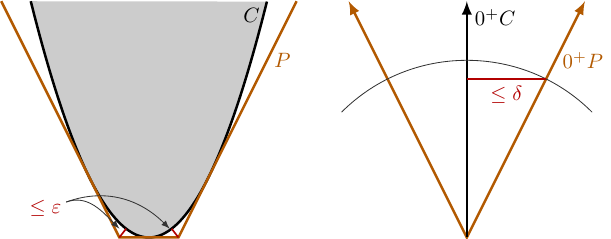}
  \decoRule
  \caption{\label{fig_ed_apprx}\textbf{Left:} The polyhedron $P$ is
    an~$\ed$-approximation of the grey set~$C$. \textbf{Right:} The
    recession cones of the sets on the left. The truncated Hausdorff
    distance between them is at most~$\delta$.}
\end{figure}

\begin{remark}
  The assumption that $P$ is line-free is equivalent to $\vertices P
  \neq \emptyset$ and hence required for condition (i) in the
  definition. Condition (iii) means that $P$ is an outer approximation
  of $C$. This is required, because otherwise the roles of $P$ and $C$
  would have to be interchanged in (i). However, it is not clear how
  to proceed with this in a meaningful fashion. The analogue of
  considering vertices of $P$ would be to consider extreme points of
  $C$ instead. The set of extreme points of $C$ may be unbounded and
  it is in general not possible to enforce the upper bound of
  $\varepsilon$. Lastly, we decided to make a distinction between
  $\varepsilon$ and $\delta$, because scales of these error measures
  may be very different depending on the sets, i.e. $\delta$ is always
  bounded from above by 1, but for $\varepsilon$ it may be useful to
  allow values larger than 1.
\end{remark}

Figure \ref{fig_ed_apprx} illustrates the definition. We will show
that an~$\ed$-approximation of a set~$C$ approximates~$C$ in a
meaningful way. To this end, we consider the Painlevé-Kuratowski
convergence, a notion of set convergence that is suitable for a
broader class of sets than convergence with respect to the Hausdorff
distance.

\begin{definition}
  A sequence~$\{C^{\nu}\}_{\nu \in \mathbb{N}}$ of subsets of~$\R^n$ is said
  to \emph{converge to~${C \subseteq \R^n}$ in the sense of
    Painlevé-Kuratowski}, denoted by~${C^{\nu} \rightarrow C}$, if the
  following equalities hold:
  \begin{equation*}
    \begin{aligned}
      C &= \left\lbrace x \in \R^n \middle\vert \begin{matrix} \text{for all open
          neighbourhoods } U \text{ of } x,\\ U \cap C^{\nu}  \neq \emptyset
          \text{ for large enough } \nu\end{matrix} \right\rbrace \\
      &= \left\lbrace x \in \R^n \middle\vert \begin{matrix} \text{for all open
        neighbourhoods } U \text{ of } x,\\ U \cap C^{\nu} \neq \emptyset
        \text{ for infinitely many }\nu\end{matrix} \right\rbrace.
    \end{aligned}
  \end{equation*}
\end{definition}

To conserve space and enhance readability we will denote by~$C^{\nu}$ the
sequence~$\{C^{\nu}\}_{\nu \in \mathbb{N}}$ as well as the specific
element~$C^{\nu}$ of this sequence whenever there is no ambiguity. The two
sets in the definition are called the inner and outer limit of~$C^{\nu}$,
respectively. Convergence in the sense of Painlevé-Kuratowski is
weaker than convergence in the Hausdorff distance, but both concepts
coincide when restricted to compact subsets, see Example \ref{EPKHC} and
\cite[pp.~131--138]{RW98}. However, for convex sets Painlevé-Kuratowski
convergence can be characterized using the Hausdorff distance.

\begin{example}[{see \cite[p.~118]{RW98}}]\label{EPKHC} %Example
                                %Painlevé-Kuratowski Hausdorff
                                %Convergence
  Consider the sequence of sets for which $C^{\nu} = \set{x,y^{\nu}}$
  for $x,y^{\nu} \in \R^n$ and $\norm{y^{\nu}} \rightarrow
  \infty$. Then $C^{\nu}$ converges in the sense of
  Painlevé-Kuratowski to the singleton $C=\set{x}$, but does not
  converge in the Hausdorff distance, because
  $\haus{C^{\nu}}{C}=\norm{x-y^{\nu}} \rightarrow \infty$.
\end{example}

\begin{theorem}[{see \cite[p.~120]{RW98}}]\label{TPKCHD}
  A sequence~$C^{\nu}$ of nonempty closed and convex sets converges to~$C$
  in the sense of Painlevé-Kuratowski if and only if there exist~$x
  \in \R^n$ and~$r_0 \in \R$, such that for all~$r \geq r_0$ it holds, that
  \begin{equation}
    \haus{C^{\nu} \cap B_r(x)}{C \cap B_r(x)} \rightarrow 0.
  \end{equation}
\end{theorem}

In geometric terms this means that a sequence of nonempty closed and
convex sets converges in the sense of Painlevé-Kuratowski if and only
if it converges in the Hausdorff distance on every nonempty compact
subset. In the remainder of this section we show
that~$\ed$-approximations provide a meaningful notion of polyhedral
approximation for unbounded sets in the sense that a sequence of
approximations converges as defined in Definition \ref{DEDA}
if~$\varepsilon$ and~$\delta$ tend to zero. To this end we need some
preparatory results. The first one yields a bound on the Hausdorff
distance between truncations of a set and truncations of
an~$\ed$-approximation.

\begin{proposition}\label{PHDA}
  Let $C \subseteq \R^n$ be nonempty closed convex and line-free
  and let~$P$ be an~$\ed$-approximation of~$C$. Then for every~$x \in
  \conv \vertices P$ and~$r \geq \varepsilon$ it holds true, that
  \begin{equation}
    \haus{P \cap B_r(x)}{C \cap B_r(x)} \leq 2\left(
      \varepsilon+\delta\left(r+\norm{x-v}\right)\right)
  \end{equation}
  for some~$v \in \conv\vertices P$. In particular, if~$\haus{P \cap
    B_r(x)}{C \cap B_r(x)}$ is attained as~$\norm{p-c}$ with~$p \in P$,
  then~$p = v + td$ for some~$d \in \recc{P}$ and~$t \geq 0$.
\end{proposition}

\begin{proof}
  Denote by $\bar{P}$, $\bar{C}$, and $V$, $P \cap B_r(x)$, $C \cap
  B_r(x)$, and $\conv \vertices P$, respectively. Since~$\bar{P},
  \bar{C}$ are nonempty convex and compact,~$\haus{\bar{P}}{\bar{C}} =
  \norm{p^*-c^*}$ for some~$p^* \in \bar{P}$ and~$c^* \in
  \bar{C}$. For~$\lambda \in [0,1]$ let~$z(\lambda) = \lambda p^* +
  (1-\lambda) x$. We distinguish two cases. First, assume~$p^* \in
  V$. Then~$z(\lambda) \in V$ for every~${\lambda \in [0,1]}$ and, due
  to (i) in Definition \ref{DEDA}, there is~${c_{\lambda} \in \bar{C}}$
  with~${\norm{z(\lambda)-c_{\lambda}} \leq
    \varepsilon}$. If~$\norm{p^*-x} \leq \varepsilon$ then
  \[ \norm{p^*-c^*} \leq \norm{p^*-c_0} \leq \norm{p^*-x} +
    \norm{x-c_0} \leq 2\varepsilon. \] If~$\norm{p^*-x} > \varepsilon$,
  set ${\lambda}^* = \frac{\varepsilon}{\norm{p^*-x}}$. Similarly, we
  have
  \[ \norm{p^*-c^*} \leq \norm{p^*-c_{{\lambda}^*}} \leq
    \norm{p^*-z({\lambda}^*)} + \norm{z({\lambda}^*)-c_{{\lambda}^*}} \leq
    2\varepsilon. \] Now, assume~${p^* \notin V}$. Then there
  exists~$\bar{\lambda} \in (0,1)$, such that~$z(\bar{\lambda}) \in V$
  and~$z(\lambda) \notin V$ for all $\lambda \in
  (\bar{\lambda},1]$. For~${\lambda \in (\bar{\lambda},1]}$,
  $z(\lambda)$ can be written as ${z(\lambda) = v_{\lambda} +
    t_{\lambda}d_{\lambda}}$ with $v_{\lambda} = \argmin \{
      \norm{z(\bar{\lambda}) - v} \;\mid\;$ $v \in \left(\{z(\lambda)\} -
        \recc{P}\right) \cap V \}$, $t_{\lambda} \geq 0$ and
  $d_{\lambda} \in \recc{P}$, $\norm{d_{\lambda}} = 1$. By the
  definition of $\ed$-approximation there exist~$c_{\lambda} \in C$
  and~$\bar{d_{\lambda}} \in \recc{C}$ such
  that~$\norm{v_{\lambda}-c_{\lambda}} \leq \varepsilon$
  and~$\norm{d_{\lambda}-\bar{d_{\lambda}}} \leq \delta$. Now we have
  \begin{equation}\label{proof_PHDA_1}
    \norm{z(\lambda) - (c_{\lambda} + t_{\lambda}\bar{d_{\lambda}})} \leq
    \norm{v_{\lambda}-c_{\lambda}} + t_{\lambda}
    \norm{d_{\lambda}-\bar{d_{\lambda}}} \leq \varepsilon +
    t_{\lambda}\delta.
  \end{equation}
  Furthermore, $t_{\lambda} =
  \norm{z(\lambda)-v_{\lambda}}$, because $\norm{d_{\lambda}} =
  1$. Hence,
  \begin{equation}\label{proof_PHDA_2}
    \begin{aligned} t_{\lambda} &\leq
      \norm{z(\lambda)-z(\bar{\lambda})} + \norm{z(\bar{\lambda}) -
        v_{\lambda}} \leq \norm{z(\lambda)-z(\bar{\lambda})} +
      \norm{z(\bar{\lambda})-v_1} \\ &\leq
      \norm{z(\lambda)-z(\bar{\lambda})} + \norm{z(\bar{\lambda})-x} +
      \norm{x -v_1} \\ &= \norm{z(\lambda)-x} + \norm{x-v_1}\\ &\leq r
      + \norm{x-v_1}.
    \end{aligned}
  \end{equation}
If $\frac{r-\varepsilon}{r+\norm{x-v_1}} < \delta \leq 1$ then \[
    \norm{p^*-c^*} \leq \norm{p^*-c_0} \leq \norm{p^*-x} +
    \norm{x-c_0} \leq r + \varepsilon \leq 2(\varepsilon + \delta
    (r+\norm{x-v_1})). \] Otherwise, the last inequality is
  violated. In this case, let~${\lambda}^* = 1-\frac{\varepsilon +
    \delta(r\norm{x-v_1})}{\norm{p^*-x}}$. If~${\lambda}^* \leq
  \bar{\lambda}$ then $z({\lambda}^*) \in V$ and there
  exists~$c_{{\lambda}^*} \in C$ with~$\norm{z({\lambda}^*) -
    c_{{\lambda}^*}} \leq \varepsilon$. Therefore,
  \[ \norm{p^*-c^*} \leq \norm{p^*-z({\lambda}^*)} +
    \norm{z({\lambda}^*)-c_{{\lambda}^*}} \leq 2 \varepsilon + \delta
    (r+\norm{x-v_1}). \] If~${\lambda}^* > \bar{\lambda}$ then,
  according to \eqref{proof_PHDA_1} and \eqref{proof_PHDA_2}, there
  exists~$c \in C$ such that $\norm{z({\lambda}^*)-c} \leq \varepsilon
  + \delta(r+\norm{x-v_1})$. Altogether this yields \[ \norm{p^*-c^*}
    \leq \norm{p^*-c} \leq \norm{p^*-z({\lambda}^*)} +
    \norm{z({\lambda}^*)-c} \leq 2 \left( \varepsilon + \delta
      (r+\norm{x-v_1})\right), \] which completes the proof.
\end{proof}

We need two more results before we can prove Theorem \ref{TPKC}
below.

\begin{lemma}\label{LBS}
  Let $C \subseteq \R^n$ be nonempty closed and convex and let there
  be sequences~$v^{\nu}$,~$r^{\nu}$ such that~$\inf_{c \in C}
  \norm{v^{\nu}-c} \rightarrow 0$, $\inf_{r \in \recc{C}}
  \norm{r^{\nu}-r} \rightarrow 0$, and $v^{\nu}+r^{\nu} \in
    B_M(x)$ for some $M \geq 0$ and $x \in \R^n$. If $C$ is
  line-free, then $\norm{v^{\nu}}$ is bounded.
\end{lemma}

\begin{proof}
  Assume that $\norm{v^{\nu}}$ is unbounded. This implies
  that~$\norm{r^{\nu}}$ is also unbounded and $\recc{C} \neq
  \{0\}$. Without loss of generality let $r^{\nu} \neq 0$ for all
  $\nu$. Then $d^{\nu} := r^{\nu}/\norm{r^{\nu}}$ is bounded and has a
  convergent subsequence. Without loss of generality we can assume
  that $d^{\nu} \rightarrow d \in \recc{C}$. We will show, that $-d
  \in \recc{C}$. Therefore let $c \in C$, $t \geq 0$ and define $y^{\nu} =
  \argmin_{y \in C} \norm{v^{\nu}-y}$. By the triangle inequality it
  holds true, that
  \begin{equation}\label{proof_LBS}
    \norm{c-r^{\nu}-y^{\nu}} \leq M + \norm{c-x} +
    \norm{v^{\nu}-y^{\nu}} =: \ols{M}^{\nu}.
  \end{equation}
  Note, that $\ols{M}^{\nu}$ is bounded from above by some $\ols{M}$, because
  $\norm{v^{\nu}-y^{\nu}} \rightarrow 0$. For every $T \geq 0$ there
  exists some $\nu_T$, such that $\norm{r^{\nu_T}} \geq T$. Let $T
  \geq t$ and define
  \[ \bar{y} := \frac{t}{\norm{r^{\nu_T}}} y^{\nu_T} + \left(
      1-\frac{t}{\norm{r^{\nu_T}}} \right) c \in C. \]
  Putting it all together one gets
  \begin{equation*}
    \begin{aligned}
      \inf_{y \in C} \norm{(c-td^{\nu_T})-y} &\leq
      \norm{(c-td^{\nu_T})-\bar{y}} \\
      &= \norm{\frac{t}{\norm{r^{\nu_T}}}c
        -td^{\nu_T}-\frac{t}{\norm{r^{\nu_T}}}y^{\nu_T}} \\
      &= \frac{t}{\norm{r^{\nu_T}}} \norm{c-r^{\nu_T}-y^{\nu_T}} \\
      &\leq \frac{t}{T} \ols{M},
    \end{aligned}
  \end{equation*}
  where the last inequality holds due to \eqref{proof_LBS} and
  boundedness of $\ols{M}^{\nu}$. Since $C$ is closed and $d^{\nu}
  \rightarrow d \in \recc{C}$, taking the limit $T \rightarrow
  +\infty$ yields that $c-td \in C$. This is a contradiction to the
  pointedness of $\recc{C}$.
\end{proof}

Every closed and line-free convex set $C$ can be written as the convex
hull of its extreme points plus its recession cone
\cite[p.~35]{Hol75}. In particular, Lemma \ref{LBS} states that the
set of convex combinations of extreme points for which a given point
in $C$ can be decomposed in such a fashion is compact. The next result
establishes a relation between extreme points of $C$ and the vertices
of an $\ed$-approximation.

\begin{proposition}\label{PEP}
  Let $C \subseteq \R^n$ be nonempty closed convex and
  line-free. For $\nu \in \mathbb{N}$ let~$P^{\nu}$ be an
  $(\varepsilon^{\nu},\delta^{\nu})$-approximation of
  $C$. If~${(\varepsilon^{\nu},\delta^{\nu}) \rightarrow (0,0)}$, then
  for every extreme point~$c$ of $C$ there exists a sequence~${x^{\nu}
    \rightarrow c}$ such that $x^{\nu} \in \conv \vertices P^{\nu}$.
\end{proposition}

\begin{proof} Since $C$ is line-free, it has at least one extreme
  point. Let $c$ be one such extreme point. Assume that for every
  sequence $x^{\nu}$ with~${x^{\nu} \in \conv \vertices P^{\nu}}$ there
  exists a $\gamma > 0$, such that~${\norm{x^{\nu}-c} > \gamma}$ for
  infinitely many $\nu$. Then, without loss of generality, there exists
  one such sequence such that~$\norm{x^{\nu}-c} > \gamma$ for every
  $\nu$ and, since $C \subseteq P^{\nu}$, $c=x^{\nu}+r^{\nu}$ for some
  $r^{\nu} \in \recc{P^{\nu}}$. By Lemma \ref{LBS} it holds that
  $\norm{x^{\nu}}$ and $\norm{r^{\nu}}$ are bounded. Then there exist
  subsequences $x^{\nu_k}$, $r^{\nu_k}$ such that
  \[ x^{\nu_k} \rightarrow x \in C, \quad r^{\nu_k} \rightarrow r \in
    \recc{C}. \] Note that $r \neq 0$, because $\norm{r^{\nu}} > \gamma$
  for all $\nu$. Finally,
  \[ c=x+r=\frac{1}{2}\left( x+2r \right) + \frac{1}{2} x. \] This is
  a contradiction to $c$ being an extreme point of $C$.
\end{proof}

We are now ready to prove the main result.

\begin{theorem}\label{TPKC}
  Let $C \subseteq \R^n$ be nonempty closed convex and
  line-free. For $\nu \in \mathbb{N}$ let~$P^{\nu}$ be
  an~$(\varepsilon^{\nu},\delta^{\nu})$-approximation of
  $C$. If~${(\varepsilon^{\nu},\delta^{\nu}) \rightarrow (0,0)}$,
  then~${P^{\nu} \rightarrow C}$ in the sense of Painlevé-Kuratowski.
\end{theorem}

\begin{proof}
  By Theorem \ref{TPKCHD} we must show that there exist $c \in \R^n$
  and $r_0 \geq 0$ such that for all $r \geq r_0$ it holds $\haus{P^{\nu} \cap B_r(c)}{C \cap B_r(c)}
  \rightarrow 0$. Let~${r \geq \max_{\nu \in
      \mathbb{N}} \varepsilon^{\nu}}$ and let $c$ be an extreme point of $C$,
  which exists, because $C$ contains no lines. By Proposition \ref{PEP}
  there exists a sequence~${x^{\nu} \rightarrow c}$ with~${x^{\nu} \in
    \conv \vertices P^{\nu}}$. Applying the triangle inequality and
  Proposition \ref{PHDA} yields
  \begin{equation*}
    \begin{aligned}
      & & &  \haus{P^{\nu} \cap B_r(c)}{C \cap B_r(c)} \\
      & \leq & & \haus{P^{\nu} \cap B_r(c)}{P^{\nu} \cap B_r(x^{\nu})}
      \\ & & + & \haus{P^{\nu} \cap B_r(x^{\nu})}{C \cap B_r(x^{\nu})}
      \\ & & + & \haus{C \cap B_r(x^{\nu})}{C \cap B_r(c)} \\
      & \leq & & \haus{P^{\nu} \cap B_r(c)}{P^{\nu} \cap B_r(x^{\nu})}
      \\ & & + & 2(\varepsilon^{\nu}+\delta^{\nu}(r+\norm{x^{\nu}-v^{\nu}}))
      \\ & & + & \haus{C \cap B_r(x^{\nu})}{C \cap B_r(c)},
   \end{aligned}
 \end{equation*}
 for some $v^{\nu} \in \conv \vertices P^{\nu}$. The first and third
 term in this sum converge to zero as~${x^{\nu} \rightarrow c}$. It
 remains to show that~$\norm{x^{\nu}-v^{\nu}}$ is
 bounded. Since $C \subseteq P^{\nu}$, the distance ${\haus{P^{\nu}
   \cap B_r(x^{\nu})}{C \cap B_r(x^{\nu})}}$ is attained as
 $\exc{P^{\nu} \cap B_r(x^{\nu})}{C \cap B_r(x^{\nu})}$. Let the
 supremum be attained by~$p^{\nu} \in P^{\nu}$. Then $p^{\nu} =
   v^{\nu} + d^{\nu}$ for some~$d^{\nu} \in \recc{P^{\nu}}$. It holds
 \[ \norm{p^{\nu}-c} \leq \norm{p^{\nu}-x^{\nu}} + \norm{x^{\nu}-c}
   \leq r + \max_{\nu \in \mathbb{N}} \norm{x^{\nu}-c} < +\infty, \]
 i.e. $v^{\nu}+d^{\nu} \in B_M(c)$ for some $M \geq 0$. Therefore, the
 sequence $\norm{v^{\nu}}$ is bounded according to Lemma
 \ref{LBS}. Hence, $\norm{x^{\nu}-v^{\nu}}$ is also bounded
 and $\haus{P^{\nu} \cap B_r(c)}{C \cap B_r(c)} \rightarrow 0$,
 which was to be proved.
\end{proof}

Theorem \ref{TPKC} justifies the definition of $\ed$-approximations,
i.e. it states that they define a meaningful notion of
approximation. We close this section with the observation that
$\ed$-approximations reduce to $\varepsilon$-approximations in the
compact case.

\begin{corollary}
Let $C \subseteq \R^n$ be a convex body and $P \subseteq \R^n$ be a
polyhedron. For $\varepsilon \geq 0$ and $\delta \in [0,1)$ the
following are equivalent.
\begin{enumerate}[(i)]
\item $P$ is an $\ed$-approximation of $C$.
\item $P \supseteq C$ and $\haus{P}{C} \leq \varepsilon$.
\end{enumerate}
\end{corollary}

\begin{proof}
  Since $C$ is compact, $\recc{C} = \{0\}$. Then
  $\thaus{\recc{P}}{\{0\}} < 1$ implies that $\recc{P} = \{0\}$,
  i.e. $P$ is compact as well, because otherwise one would have
  $\thaus{\recc{P}}{\{0\}} = 1$. Therefore, $P$ is the convex hull of
  its vertices. Because $P \supseteq C$, $\haus{P}{C}$ is attained as
  $\exc{P}{C}$. But $\exc{P}{C}$ is attained in a vertex of $P$,
  i.e. $\exc{P}{C} = \exc{\vertices P}{C}$. Hence, $\haus{P}{C} \leq
  \varepsilon$. On the other hand, if $\haus{P}{C} \leq \varepsilon$,
  then $P$ must be compact by Proposition \ref{PHDRC}. Then $\recc{P}
  = \recc{C}$ and $P$ is an $(\varepsilon,0)$-approximation of $C$ and
  in particular an $\ed$-approximation.
\end{proof}

\section{An Algorithm for the Polyhedral Approximation of Unbounded
  Spectrahedra}\label{SEC_5}
In this section we present an algorithm for computing
$\ed$-approximations of closed convex and line-free sets $C$ whose
interior is nonempty. We also prove correctness and finiteness of the
algorithm. The algorithm employs a cutting scheme, a procedure for
approximating convex bodies by polyhedra that is introduced in
\cite{Kam92}. A cutting scheme is an iterative algorithm that computes
a sequence of polyhedral outer approximations by successively
intersecting the approximation with new halfspaces. In doing so,
vertices of the current approximation are cut off, hence the name. The
calculation of these halfspaces is explained in Proposition
\ref{PP2D2} below.

Since we are dealing with unbounded sets, we pursue the idea to reduce
computations to certain compact sets and then apply a cutting
scheme. Furthermore, we have to be able to assess the set
$\recc{C}$. Since this is difficult in the general case, we only
consider sets $C$ that are spectrahedra, because a representation of
the recession cone is readily available.

Throughout this section we consider the following semidefinite
programs related to a closed spectrahedron $C = \{x \in \R^n \mid
\A(x) \mgeq 0 \}$ with nonempty interior. For a direction $w \in \R^n
\setminus \{0\}$ consider
\begin{equation}\tag{P$_1$($w$)}\label{P1}
\begin{aligned}
\max \; w\T x \quad & \text{s.t.} \quad & \A(x) \mgeq 0.
\end{aligned}
\end{equation}
Solving \eqref{P1} is equivalent to determining the maximal shifting
of a hyperplane with normal $w$ within $C$. The following result is
well known in the literature, see e.g. \cite[Corollary 14.2.1]{Roc70}.

\begin{proposition}\label{PP1}
  For every $w \in \interior \polar{(\recc{C})}$ an optimal solution
  to \eqref{P1} exists.
\end{proposition}

The second problem we consider is
\begin{equation}\tag{P$_2$($v,d$)}\label{P2}
\begin{aligned}
\min \; t \quad & \text{s.t.} \quad & \A(x) &\mgeq 0 \\
& & x-v-td &= 0,
\end{aligned}
\end{equation}
where $v \in \R^n$ and $d \in \R^n\setminus\{0\}$. Solving \eqref{P2}
can be described as the task of determining the maximum distance on
can move in direction $d$ starting at point $v$ until the set $C$ is
reached. If this distance is finite and $v \notin C$, then a solution
to \eqref{P2} yields a point on the boundary of $C$, namely one of the
points that are obtained by intersecting the boundary of $C$ with the
affine set $\{v+td \mid t \in \R \}$. The Lagrangian dual problem of
\eqref{P2} is
\begin{equation}\tag{D$_2$($v,d$)}\label{D2}
\begin{aligned}
\max \; -\A(v) \cdot U \quad & \text{s.t.} \quad & A_i \cdot U &= w_i,
\; \forall i \in \{1,\dots,n\} \\
& & d\T w &= 1 & \\
& & U &\mgeq 0. &
\end{aligned}
\end{equation}
Solutions to \eqref{P2} and \eqref{D2} give rise to a supporting
hyperplane of $C$ as described in the next proposition.

\begin{proposition}\label{PP2D2}
Let $v \notin C$ and set $d=c-v$ for some $c \in \interior C$. Then
solutions $(x^*,t^*)$ to \eqref{P2} and $(U^*,w^*)$ to \eqref{D2}
exist. Moreover, $w^{*\mathsf{T}} x \geq w^{*\mathsf{T}}v+t^*$ for all $x \in C$ and
equality holds for $x=x^*$.
\end{proposition}

\begin{proof}
Without loss of generality we can assume that $\interior C = \{x \in
\R^n \mid \A(x) \mg 0\}$, see \cite[Corollary 5]{GR95}. Then $(c,1)$
is strictly feasible for \eqref{P2}, which is the well known Slater's
constraint qualification in convex optimization. Since, $v \notin C$
and by convexity the first constraint is violated whenever $t \leq
0$. Since $C$ is closed, an optimal solution $(x^*,t^*)$ to \eqref{P2}
with $t^* \in [0,1]$ exists. Slater's constraint qualification now
implies strong duality, i.e. an optimal solution $(U^*,w^*)$ to
\eqref{D2} exists and the optimal values conincide. Next, let $x \in
C$ and observe that
\begin{equation*}
  \begin{aligned}
    w^{*\mathsf{T}}x-w^{*\mathsf{T}}v-t^* &=
    \sum_{i=1}^n x_i (A_i \cdot U^*) -w^{*\mathsf{T}}v-t^* \\
    &= \Ab(x) \cdot U^* -\sum_{i=1}^n v_i (A_i \cdot U^*) -t^* \\
    &= \Ab(x) \cdot U^* - \Ab(v) \cdot U^* + \A(v) \cdot U^* \\
    &= \Ab(x) \cdot U^* + A_0 \cdot U^* \\
    &= \A(x) \cdot U^* \geq 0.
  \end{aligned}
\end{equation*}
The third equality holds due to strong duality. Lastly, for $x=x^*$ we
have equality, because $x^* = v+t^*d$ and $w^{*\mathsf{T}}d = 1$.
\end{proof}

We want to describe the functioning of the algorithm geometrically
before we present the details in pseudo code. The method consists of
two phases. In the first phase an initial polyhedron $P_0$, such that
$P_0 \supseteq C$ and $\thaus{\recc{P_0}}{\recc{C}} \leq \delta$, is
constructed as follows: For $w \in \interior \polar{(\recc{C})}$ the
set
\begin{equation}\label{AEBRC} %Algorithm Explanation Basis Recession
                             %cone C
M := \{x \in \R^n \mid w\T x = -(1+\delta)\}
\end{equation}
is a compact basis of $\recc{C}$, i.e. $\recc{C} = \cone M$. We use a
cutting scheme to compute a polyhedral $\delta$-approximation
$\ols{M}$ of $M$ with $M \subseteq \interior \ols{M}$. If in
\eqref{AEBRC} we set $\norm{w}=1$, then
\begin{equation}
K := \cone \ols{M}
\end{equation}
is a polyhedral cone with $\thaus{K}{\recc{C}} \leq \delta$. Next, we
need to construct a polyhedron $P_0$ with recession cone $K$ that
contains $C$. To this end we compute a $H$-representation $(R,0)$ of
$K$ and solve (P$_1$($r$)) for every row $r$ of $R$, that is for every
normal of supporting hyperplanes that define $K$. Note, that a
solution always exists, because $r \in \interior \polar{(\recc{C})}$
by construction. For a solution $x_r^*$ to (P$_1$($r$)) the set
\begin{equation}\label{AESH} %Algorithm Explanation Supporting Halfspace
\{x \in \R^n \mid r\T x = r\T x_r^*\}
\end{equation}
is a hyperplane that supports $C$ in $x_r^*$. For the initial
approximation we then set
\begin{equation}\label{EQIA} %EQuation Initial Approximation
P_0 = \bigcap_r \{x \in \R^n \mid r\T x \leq r\T x_r^*\}.
\end{equation}
Clearly, it holds that $\recc{P_0} = K$ and that $P_0$ has at least one
vertex, because $K$ is pointed.

In the second phase of the algorithm $P_0$ is refined by successively
cutting of vertices until all vertices are within distance of at most
$\varepsilon$ from $C$. This is achieved by iteratively intersecting
$P_0$ with halfspaces that support $C$ in some point of its
boundary. To guarantee finiteness of the algorithm we retreat with
the computations to compact subsets $\ols{P_0}$ of $P_0$ and $\ols{C}$
of $C$ obtained by intersecting the sets with a halfspace, namely
\begin{equation}
  \ols{P_0} = P_0 \cap \{x \in \R^n \mid w\T x \geq \min\{w\T x \mid x
  \in \conv\ext\left(C+K\right)\}- \varepsilon \}
\end{equation}
and
\begin{equation}
  \ols{C} = C \cap \{x \in \R^n \mid w\T x \geq \min\{w\T x \mid x \in
  \conv\ext\left(C+K\right)\}- \varepsilon \},
\end{equation}
where $w$ is the same as in \eqref{AEBRC}. How the above halfspace can
be computed will be discussed in Variant \ref{VAEDA} below. A cutting
scheme is then applied to compute an outer $\varepsilon$-approximation
$\ols{P}$ of $\ols{C}$. Finally, an $\ed$-approximation of $C$ is
obtained as
\begin{equation}
  \ols{P} + K.
\end{equation}
We describe the aforementioned cutting scheme due to \cite{Kam92} that
is used in the computation of an $\ed$-approximation for the special
case of spectrahedral sets in Algorithm \ref{ACS}.

\begin{algorithm}
  \DontPrintSemicolon
  \KwData{Matrix $\A(x)$ representing a compact spectrahedron $C$ with
    nonempty interior, error tolerance $\varepsilon$}
  \KwResult{$V$-representation of a polyhedral
    $\varepsilon$-approximation $P$ of $C$}
  Compute solutions $x_w^*$ of \eqref{P1} for $w \in
  \{-e,e_1,\dots,e_n\}$\;
  $P \isassigned \set{x \in \R^n \mid \forall w \in
    \set{-e,e_1,\dots,e_n}: w\T x \leq w\T x_w^*}$\;
  $c \isassigned \frac{1}{n+1}\sum_w x_w^*$ \tcp*{interior point
    of $C$}
  $\kappa \isassigned +\infty$\;
  \While{$\kappa > \varepsilon$}{
    Compute the set $\vertices P$\;
    \For{every $v \in \vertices P$}{
      $d \isassigned c-v$\;
      Compute solutions $(x_v^*,t_v^*)$ and $(U_v^*,w_v^*)$ to
      \eqref{P2} and \eqref{D2}\;
    }
    $\bar{v} \isassigned \argmax \set{t_v^*\norm{c-v} \mid v \in \vertices P}$\;
    $P \isassigned P \cap \set{x \in \R^n \mid
      w_{\bar{v}}^{*\mathsf{T}}x \geq w_{\bar{v}}^{*\mathsf{T}}\bar{v}
      +t_{\bar{v}}^*}$\;
    $\kappa \isassigned t_{\bar{v}}^*\norm{c-\bar{v}}$\;
  }
  Compute $\vertices P$ and \KwRet{}
  \caption{A cutting scheme algorithm for spectrahedra}\label{ACS}
\end{algorithm}

The vectors $e$ and $e_i$, $i=1,\dots,n$, in line 1 denote the vector
in $\R^n$ with components all equal to one and the $i$-th unit vector,
respectively. Since $C$ is compact, it holds that $\interior
\polar{(\recc{C})} = \R^n$. Therefore, Proposition \ref{PP1} implies
that optimal solutions $x_w^*$ in line 1 always exist. Note, that
$\kappa$ in line 12 is an upper bound on the Hausdorff distance
between $P$ and $C$ due to the following observation. The Hausdorff
distance between $P$ and $C$ is attained in a vertex of $P$, because
$C \subseteq P$. Since the part $x_v^*$ of an optimal solution of
\eqref{P2} is an element of the boundary of $C$, we conclude $\inf_{x
\in C} \norm{x-v} \leq \norm{x_v^*-v} = t_v^*\norm{c-v}$ for every $v
\in \vertices P$. Hence, the algorithm terminates with $\haus{P}{C}
\leq t_{\bar{v}}^*\norm{c-\bar{v}} \leq \varepsilon$. For the special
class of spectrahedral sets the cutting scheme algorithm terminates
after finitely many steps. This is proved in \cite[Theorem
4.38]{Cir19}.

\begin{remark}
  As mentioned at the beginning of this section, Algorithm \ref{ACS}
  falls into the class of cutting scheme algorithms. In \cite{Kam92}
  convergence properties for a similar class of algorithms, called
  Hausdorff schemes, are established. The authors define a Hausdorff
  scheme as a polyhedral approximation algorithm fulfilling the
  condition
  \[ \haus{P^k}{P^{k+1}} \geq \gamma \haus{P^k}{C} \] for a positive
  constant $\gamma$ in every iteration. Here, $P^k$ denotes the
  polyhedral approximation obtained in iteration $k$. They show, that
  for every $\varepsilon > 0$ there exists an index $k_0$, such that for
  all $k \geq k_0$
  \[ \haus{P^k}{C} \leq (1+\varepsilon)\Gamma(C,n)\frac{1}{k^{1/(n-1)}} \]
  holds for a positive constant $\Gamma(C,n)$. Note, that if in step 8 of
  Algorithm \ref{ACS} we were able to choose $d$ such that $\kappa$
  in line 12 was equal to $\haus{P}{C}$, then our algorithm would be a
  Hausdorff scheme with constant $\gamma=1$ and the bound would hold.
\end{remark}

\begin{remark}
  Algorithm \ref{ACS} uses similar techniques as the supporting
  hyperplane method introduced in \cite{Vei67} for the maximization of
  a linear function subject to quasiconvex constraints. The supporting
  hyperplane method also constructs a sequence of polyhedral outer
  approximations of a convex body by successively introducing
  supporting hyperplanes. In order to find the corresponding boundary
  points, the same geometric idea is employed, i.e. moving from
  vertices of the current approximation towards an interior point
  until the boundary is met. However, the algorithms differ in
  multiple aspects. Firstly, in each iteration we choose the vertex with
  the largest distance to the set $C$ with respect to the direction
  $d$, while in \cite{Vei67} the vertex that realizes the smallest
  objective function value is chosen. Secondly, we do not assume $C$
  to have a continuously differentiable boundary. In particular, if
  $\A(x)$ is a diagonal matrix, then $C$ is a polyhedron. Therefore,
  our algorithm can handle a larger class of sets. Finally, the
  supporting hyperplane method approximates the set only in a
  neighbourhood of the optimal solution to the underlying optimization
  problem, while we are interested in an approximation of the whole
  set $C$.
\end{remark}

We are now prepared to present Algorithm \ref{AEDA}, an algorithm for
the computation of $\ed$-approximations of closed and line-free
spectrahedra with nonempty interior.

\begin{algorithm}
  \DontPrintSemicolon
  \KwData{Matrix $\A(x)$ representing an unbounded closed and
    line-free spectrahedron $C$ with nonempty interior, error tolerances
    $\varepsilon$ and $\delta$}
  \KwResult{$V$-representation of an $\ed$-approximation $P$ of
    $C$}
  $w \isassigned \left(-A_1 \cdot I,\dots,-A_n \cdot I\right)\T /
  \norm{\left(-A_1 \cdot I,\dots,-A_n \cdot I\right)\T}$
  \tcp*{interior point of $\polar{(\recc{C})}$}
  $M \isassigned \recc{C} \cap \set{x \in \R^n \mid w\T
    x=-(1+\delta)}$\;
  Compute a $\delta/2$-approximation $\ols{M}$ of $M$ according to
  Algorithm \ref{ACS}\;
  $\ols{M} \isassigned \ols{M}+\set{x \in \R^n \mid \norm{x}_1 \leq
    \frac{\delta}{2}}$\;
  $K \isassigned \cone \vertices \ols{M}$\;
  $\ols{C} \isassigned C \cap \{x \in \R^n \mid w\T x \geq \min\{w\T x
  \mid x \in \conv\ext\left(C+K\right)\} - \varepsilon\}$\;
  Compute an $\varepsilon$-approximation $\ols{P}$ of $\ols{C}$
  according to Algorithm \ref{ACS}\;
  $P \isassigned \ols{P} + K$\;
  \KwRet{a $V$-representation of $P$}
  \caption{An algorithm for $\ed$-approximations of spectrahedra}\label{AEDA}
\end{algorithm}

Steps 6 and 13 in Algorithm \ref{ACS} and 5 and 9 in Algorithm
\ref{AEDA} require the computation of a $V$-representation from an
$H$-representation or vice versa. These problems are known as vertex
enumeration and facet enumeration, respectively, and are difficult
problems on their own. It is beyond the scope of this paper to discuss
these problems in more detail. Therefore, we only point out that there
exist toolboxes that are able to perform these tasks numerically, such
as \texttt{bensolve tools} \cite{LW16, CLW18}. In practice, however,
the computations often become infeasible in dimensions three and
higher when the number of halfspaces defining the polyhedron is
large. It is also known, that vertex enumeration for unbounded
polyhedra is NP-hard, see \cite{KBBEG08}. Thus, since vertex
enumeration has to be performed in every iteration of Algorithm
\ref{ACS} and for the unbounded polyhedron $P$ in step 9 of Algorithm
\ref{AEDA}, one cannot expect the algorithms to be computationally
efficient.

In order to prove that Algorithm \ref{AEDA} works correctly and is
finite, we need the following preliminary result.

\begin{proposition}\label{PEB}
  Let $C \subseteq \R^n$ be a closed convex set and $K \subseteq \R^n$
  be a closed convex cone such that $\recc{C}\setminus\{0\} \subseteq
  \interior K$. Then $\ext\left(C+K\right)$ is bounded.
\end{proposition}
\begin{proof}
  We may assume that $C+K$ is pointed. Otherwise,
  $\ext\left(C+K\right) = \emptyset$ and the statement is
  vacuous. Note that $\ext\left(C+K\right) \subseteq \ext C$ because
  for every $x \in C$ and $d \in K \setminus \{0\}$ it is true that
  \[
    x+d = \frac{1}{2}x + \frac{1}{2}\left(x+2d\right),
  \]
  i.e. $x+d$ is not an extreme point of $C+K$. Now, assume that
  $\ext\left(C+K\right)$ is unbounded. Then $\ext C$ is unbounded as
  well. Let $\{x_k\}_{k \in \mathcal{N}}$ be an unbounded sequence of
  extreme points of $C+K$. Without loss of generality we assume that
  $\{\norm{x_k}\}_{k \in \mathcal{N}}$ is strictly monotonically
  increasing. If this condition is not satisfied, we can pass to a
  suitable subsequence. Define the sequence of radial projections of
  $x_k-x_1$ as
  \[
    \{d_k\}_{k \geq 2} = \left\lbrace\frac{x_k-x_1}{\norm{x_k-x_1}}\right\rbrace_{k
      \geq 2}.
  \]
  Since $d_k \in B_1(0)$ for all $k \geq 2$, it has a convergent
  subsequence. Again, without loss of generality, assume $\{d_k\}_{k
  \geq 2}$ is itself convergent with limit $\bar{d}$. According to
  \cite[Theorem 8.2]{Roc70} it holds $\bar{d} \in
  \recc{\left(C-\set{x_1}\right)} = \recc{C}$. Since $\recc{C}
  \setminus \set{0} \subseteq \interior K$ and $\norm{\bar{d}} = 1$,
  there exists some $k_0 \in \mathcal{N}$ such that $d_k \in K$ for
  all $k \geq k_0$. This implies $x_k-x_1 \in K$ for all $k \geq k_0$
  as $K$ is a cone. Therefore, $x_k \in \{x_1\}+K$ for all $k \geq
  k_0$. However, this contradicts the assumption that $x_k \in
  \ext\left(C+K\right)$ for all $k \in \mathcal{N}$.
\end{proof}

\begin{theorem}
  As inputs for Algorithm \ref{AEDA} let $\varepsilon, \delta > 0$ and
  the spectrahedron defined by the matrix $\A(x)$ be closed convex
  line-free and with nonempty interior. Then Algorithm \ref{AEDA} works
  correctly, i.e. if it terminates it returns an $\ed$-approximation of
  $C$ according to Definition \ref{DEDA}.
\end{theorem}

\begin{proof}
  Since $C$ is closed and does not contain any lines, its recession
  cone is also closed and pointed. This implies that
  $\polar{(\recc{C})}$ has nonempty interior, see
  e.g. \cite[p. 53]{BV04}. The direction $w$ defined in line 1 is an
  element of $\polar{(\recc{C})}$ according to \ref{EPCS}. Note, that $w
  \neq 0$, because $\recc{C} \neq \set{0}$ and the pointedness of
  $\recc{C}$ implies that the matrices $A_1,\dots,A_n$ are linearly
  independent \cite[Lemma 3.2.9]{Net11}. To see that $w$ is indeed from
  the interior of $\polar{(\recc{C})}$ observe that for every $x \in
  \recc{C}\setminus \{0\}$ it holds, that
  \[ w\T x = -\sum_{i=1}^n x_i (A_i \cdot I) = - \left(\sum_{i=1}^n
      x_iA_i\right) \cdot I = - \Ab(x) \cdot I < 0. \]
  The last inequality holds, because at least one eigenvalue of
  $\Ab(x)$ is positive. The set $M$ defined in line 2 is compact,
  because $w \in \interior \polar{(\recc{C})}$. Note, that $M$ is not
  full-dimensional, however, treating its affine hull as the ambient
  space, $M$ is a valid input for Algorithm \ref{ACS} in line 3. By
  enlarging $\ols{M}$ in line 4 it remains polyhedral as the Minkowski
  sum of polyhedra. The cone $K$ is then polyhedral and it satisfies
  $\recc{C}\setminus \{0\} \subseteq \interior K$ and
  $\thaus{K}{\recc{C}} \leq \delta$. The first assertion is immediate
  from the observation that $\recc{C} = \cone M$ and $M \subseteq
  \interior \ols{M}$. Secondly, it is true that $\norm{x} \geq 1+\delta$
  for every $x$ satisfying $w\T x = -(1+\delta)$. Therefore, $\norm{x}
  \geq 1$ for every $x \in \ols{M}$ due to the construction of the
  set. Assume $\thaus{K}{\recc{C}}$ is attained as $\norm{k-c}$ and let
  $\alpha$ be chosen such that $\alpha k \in \ols{M}$, in particular
  $\alpha \geq 1$. Then we obtain the second claim by the following
  observation:
  \begin{equation*}
    \begin{aligned}
      \thaus{K}{\recc{C}} = \norm{k-c} &\leq \alpha \norm{k-c}\\
      &\leq \inf_{\substack{c \in \recc{C},\\ w\T c=-(1+\delta)}} \norm{\alpha k -
        c}\\
      &\leq \haus{\ols{M}}{M}\\
      &\leq \delta.
    \end{aligned}
  \end{equation*}
  Note that $K$ is pointed because $0 \notin \ols{M}$.  Since
  $\recc{C} \setminus \set{0} \subseteq \interior K$,
  $\ext\left(S+K\right)$ is bounded according to Proposition
  \ref{PEB}. Therefore, its convex hull is bounded as well. Moreover,
  it is nonempty because $C+K$ is closed according to \cite[Corollary
  9.1.1]{Roc70} and line-free due to $K$ being pointed. Thus,
  \cite[Corollary 18.5.2]{Roc70} guarantees the existence of an
  extreme point. Moreover, the set $\conv\ext\left(C+K\right)$ is
  closed, cf. \cite{Hol75}. Together this implies that $\inf\set{w\T
  x}{x \in \conv\ext\left(C+K\right)}$ is finite and attained,
  i.e. line 6 of the algorithm is well-defined. The additional shift
  of $-\varepsilon$ on the right hand side of the halfspace ensures
  that the intersection with $C$ has nonempty interior. Furthermore,
  $\ols{C}$ is compact, see \cite[Theorem 12]{GIMT13}. Hence, it is a
  valid input to Algorithm \ref{ACS} in line 7 and an
  $\varepsilon$-approximation $\ols{P}$ of $\ols{C}$ is computed
  correctly.
  
  Now, we show that $P = \ols{P} + K$ is an $\ed$-approximation of
  $C$. Since $\ols{P}$ is compact, it holds $\recc{P} = K$ and $P$ is
  pointed. We have already demonstrated $\thaus{K}{\recc{C}} \leq
  \delta$. As $\vertices P \subseteq \vertices \ols{P}$, it holds
  \[
    \exc{\vertices P}{C} \leq \exc{\vertices \ols{P}}{C} \leq
    \exc{\vertices \ols{P}}{\ols{C}} \leq \varepsilon.
  \]
  It remains to show $C \subseteq P$. From the fact that
  $\conv\ext\left(C+K\right) \subseteq \ols{C}$ one obtains the
  decomposition
  \[
    C+K = \left(\left(C+K\right) \cap H^+\right) + K,
  \]
  where $H^+$ denotes the halfspace utilized in line 6, see
  \cite[Corollary 2]{GIMT13}. Moreover, it is easy to verify that
  \[
    \left(\left(C+K\right) \cap H^+\right) + K = \ols{C}+K
  \]
  is true because the normal vector $w$ of $H^+$ satisfies $w \in
  K^{\circ}$. We conclude
  \[
    C \subseteq C+K = \ols{C}+K \subseteq \ols{P}+K = P.
  \]
  This completes the proof of correctness.
\end{proof}

\begin{corollary}
  Algorithm \ref{AEDA} terminates after finitely many steps.
\end{corollary}

\begin{proof}
  This is a consequence of the finiteness of Algorithm \ref{ACS},
  see \cite[Theorem 4.38]{Cir19}. Therefore, the executions of
  Algorithm \ref{ACS} in lines 3 and 7 of Algorithm \ref{AEDA}
  terminate after finitely many steps, which implies that Algorithm
  \ref{AEDA} itself is finite.
\end{proof}

The difficulty of Algorithm \ref{AEDA} is the determination of the
halfspace
\begin{equation}\label{eqtn_hs}
  \left\lbrace x \in \R^n \mid w\T x \geq \min\left\lbrace w\T x \mid x
  \in \conv\ext\left( C+K \right)\right\rbrace -
  \varepsilon\right\rbrace
\end{equation}
in line 6. The reason is that, although $\min\set{w\T x}{x \in
\conv\ext\left(C+K\right)}$ is a convex program, a representation of
$\conv\ext\left(C+K\right)$ as a spectrahedron or a more general
description in terms of convex functions is not readily available. In
fact, it is an open question whether $\conv\ext\left(C+K\right)$ is a
spectrahedron in the first place.

So far, the only knowledge we have about $\conv\ext\left(C+K\right)$
is acquired from Proposition \ref{PEB}, which implies its compactness
and the existence of the halfspace \eqref{eqtn_hs}. In order to deal
with this limitation from a computational perspective, we suggest a
modification of Algorithm \ref{AEDA}. Note that in line 6 of the
algorithm it suffices to intersect $C$ with a halfspace $H^+$ such
that their intersection is bounded and the containment
\[
  \conv\ext\left(C+K\right) \subseteq H^+
\]
is satisfied, cf. \cite[Corollary 2]{GIMT13}. The following variant of
Algorithm \ref{AEDA} computes such a halfspace iteratively. It
replaces lines 6--8 of the original algorithm.

\begin{algorithm}[H]
  \DontPrintSemicolon
  Compute the set $R$ of extreme directions of $K^{\circ}$\;
  \For{every $r \in R$}{
    Compute a solution $x^*_r$ of (P$_1$($r$))\;
  }
  $\alpha \isassigned \min\left(\{w\T x^*_r \mid r \in R\} \cup
  \{0\}\right) - \varepsilon$\;
  \Repeat{$C \subseteq P$}{
    $\ols{C} \isassigned C \cap \{x \in \R^n \mid w\T x \geq \alpha\}$\;
    Compute an $\varepsilon$-approximation $\ols{P}$ of $\ols{C}$
  according to Algorithm \ref{ACS}\;
  $\alpha \isassigned 2\alpha$\;
  $P \isassigned \ols{P} + K$\;
  }
  \caption{Variant of lines 6--8 of Algorithm \ref{AEDA}}\label{VAEDA}
\end{algorithm}

After computing the set of extreme directions of $K^{\circ}$ via
vertex enumeration, problem (P$_1$($r$)) is solved for every extreme
direction $r$ of $K^{\circ}$. Solutions $x^*_r$ exist according to
Proposition \ref{PP1}. These solutions give rise to an initial
halfspace $H^+$ with normal being the direction $w$ computed in line 1
of Algorithm \ref{AEDA} and right and side
\[
  \alpha = \min\left(\{w\T x^*_r \mid r \in R\} \cup \{0\}\right)
  -\varepsilon.
\]
A compact subset $\ols{C}$ of $C$ is obtained as the intersection of
$C$ and $H^+$. It has nonempty interior because $x^*_r \in \interior
H^+$ and $\interior C \neq \emptyset$. Moreover, $\alpha < 0$. Now,
Algorithm \ref{ACS} is used to compute a polyhedral outer
approximation $\ols{P}$ of $\ols{C}$ with tolerance $\varepsilon$ and
it is checked whether $P = \ols{P}+K$ is an $\ed$-approximation of $C$
by verifying the containment $C \subseteq P$. If the containment
holds, the algorithm is terminated and $P$ is returned as a
solution. Otherwise, a new compact subset is obtained by doubling the
value of $\alpha$ in the definition of $H^+$, which corresponds to a
shift of $H^+$ in the direction $-w$, and the approximation is
repeated.

The containment $C \subseteq P$ can easily be verified using
semidefinite programming. Suppose $(A,b)$ is an H-representation of
$P$ for $A \in \R^{m \times n}$ and $b \in \R^m$. Let $a^i \in \R^n$
denote the $i$-th row of $A$. Then $C \subseteq P$ if and only if
$\sup_{x \in C}a^{i\mathsf{T}}x \leq b_i$ for every
$i=1,\dots,m$. This follows from the fact that the hyperplane $\{x \in
\R^n \mid a^{i\mathsf{T}}x \leq \sup_{x \in C} a^{i\mathsf{T}}x\}$ is
a supporting hyperplane of $C$ whenever $a^i \neq 0$ and the right
hand side is finite. The value $\sup{x \in C} a^{i\mathsf{T}}x$ is
obtained by solving problem (P$_1$($a^i$)). Thus, the containment in
line~10 can be verified by computing an H-representation of $P$ and
solving $m$ semidefinite programs. Since the set
$\conv\ext\left(C+K\right)$ is bounded according to Proposition
\ref{PEB}, it will eventually be contained in $\ols{C}$ and the
algorithm terminates.

We close this section by illustrating Algorithm \ref{AEDA} with the
following two examples.

\begin{example}\label{exmpl_algrthm}
  Consider the spectrahedron $C \subseteq \R^2$ defined by the
  matrix inequality
  \[
    \begin{pmatrix}
      x_1 & 1 & 0 & 0 \\
      1 & x_2 & 0 & 0 \\
      0 & 0 & 1 & x_1 \\
      0 & 0 & x_1 & x_2
    \end{pmatrix}
    \mgeq 0.
  \]
  It is the intersection of the epigraphs of the functions $x \mapsto
  1/x$, restricted to the positive real line, and $x \mapsto x^2$. We
  use the solver \texttt{SDPT3} \cite{TTT99, TTT03} and the software
  \texttt{bensolve tools} \cite{LW16, CLW18} to solve the semidefinite
  subproblems and perform vertex and facet enumeration,
  respectively. The algorithm is implemented in \texttt{GNU Octave}
  \cite{EBHW21}. Figure \ref{fig_exmpl} shows the polyhedral
  approximations of $C$ at different stages of Algorithm \ref{AEDA} for
  the tolerances $\ed = (0.1, 0.1)$.

  \begin{figure}
    \centering
    \begin{subfigure}[t]{.45\textwidth}
      \centering
      \includegraphics[]{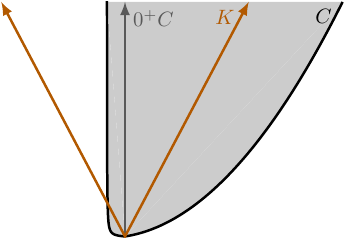}
      \caption{\label{fig_exmpl_a} The approximate recession cone $K$
        as computed in lines 1-5.}
    \end{subfigure}
    \hspace{.5cm}
    \begin{subfigure}[t]{.45\textwidth}
      \centering
      \includegraphics[]{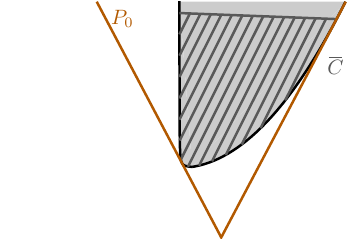}
      \caption{\label{fig_exmpl_b} $P_0$ is the initial approximation
        of $C$ according to \eqref{EQIA}. The grey hatched area is the
      compact subset $\ols{C}$ of $C$ computed in line 6.}
    \end{subfigure}

    \ \\
    \ \\
    
    \begin{subfigure}[t]{.45\textwidth}
      \centering
      \includegraphics[]{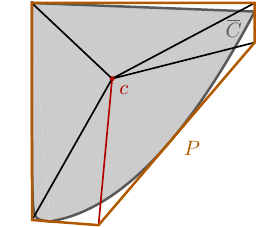}
      \caption{\label{fig_exmpl_c} One iteration of Algorithm
        \ref{ACS} applied to $\ols{C}$ in line 7. The line segments
        connecting vertices of the current approximation $P$ with the
        interior point $c$ correspond to the directions for which the
        subproblems \eqref{P2} are solved. The vertex from which the
        red line segment is emanating yields the largest distance to
        $\ols{C}$.} 
    \end{subfigure}
    \hspace{.5cm}  
    \begin{subfigure}[t]{.45\textwidth}
      \centering
      \includegraphics[]{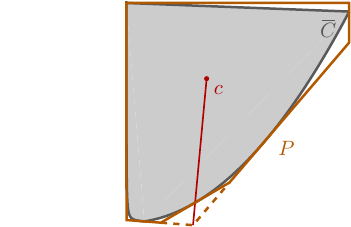}
      \caption{\label{fig_exmpl_d} The updated approximation after a
        vertex of the approximation in Figure \ref{fig_exmpl_c} is cut
      off. The red line segment and the new facet of $P$ intersect the
      boundary of $\ols{C}$ in the same point.}
    \end{subfigure}

    \ \\
    \ \\
    
    \begin{subfigure}[t]{.45\textwidth}
      \centering
      \includegraphics[]{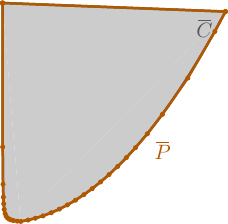}
      \caption{\label{fig_exmpl_e} $\ols{P}$ is the computed
        $\varepsilon$-approximation of $\ols{C}$ after Algorithm
        \ref{ACS} terminates in line 7. The orange dots are the
        vertices of $\ols{C}$.}
    \end{subfigure}
    \hspace{.5cm}    
    \begin{subfigure}[t]{.45\textwidth}
      \centering
      \includegraphics[]{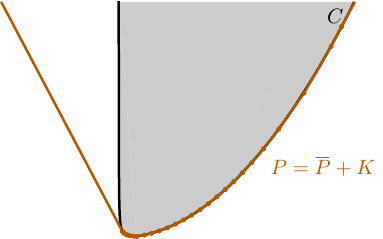}
      \caption{\label{fig_exmpl_f} The $\ed$-approximation $P$ of $C$
        that is computed by the algorithm.}
    \end{subfigure}
    \decoRule
    \caption{\label{fig_exmpl} The computed polyhedra at different
      steps of Algorithm \ref{AEDA} for the error tolerances $\ed =
      (0.1,0.1)$. Figures \ref{fig_exmpl_a} and
      \ref{fig_exmpl_c}-\ref{fig_exmpl_f} are scaled vertically by a
      factor of 0.077, Figure \ref{fig_exmpl_b} by a factor of 0.054
      for visibility.}
  \end{figure}

  \begin{table}
  \centering
  \caption{\label{tbl_exmpl} Computational results for Example
    \ref{exmpl_algrthm} and different values of $\varepsilon$ and
    $\delta$. Every cell shows the number of solved semidefinite
    subproblems, the number of vertices of the polyhedral approximation,
    as well as the elapsed CPU time.}
  \begin{tabular}{rccc}
    \toprule
    \diagbox{\boldmath{$\varepsilon$}}{\boldmath{$\delta$}} & 0.1 & 0.15 & 0.2 \\
    \midrule
    \addlinespace[1ex]
    0.1 & \makecell{603 \\ 26 \\ 230.98} & \makecell{359 \\ 20 \\ 139.92} & \makecell{261 \\ 16 \\ 100.22} \\
    \addlinespace[1ex]
    0.3 & \makecell{239 \\ 15 \\ 93.57} & \makecell{161 \\ 12 \\ 62.79} & \makecell{99 \\ 9 \\ 39.15} \\
    \addlinespace[1ex]
    0.5 & \makecell{198 \\ 14 \\ 76.38} & \makecell{99 \\ 9 \\ 38.76} & \makecell{99 \\ 9 \\ 39.85} \\
    \bottomrule
  \end{tabular}
\end{table}

  Computational results for different values of $\varepsilon$ and
  $\delta$ are presented in Table \ref{tbl_exmpl}. It can be seen that
  the number of subproblems that have to be solved is larger than the
  number of vertices the polyhedral approximation has. The reason is
  that one instance of \eqref{P2} is solved for every vertex of the
  current approximation in every iteration of Algorithm \ref{ACS}, but
  only one of these vertices is cut off. Moreover, the number of
  solved subproblems grows quickly as $\varepsilon$ decreases, because
  more iterations of Algorithm \ref{ACS} are needed to reach the
  desired accuracy and the number of solved subproblems grows with
  every iteration. Since the recession cone of $C$ is just a ray and
  easy to approximate, most of the computational effort is put into
  approximating $\ols{C}$ in line 7. However, for fixed $\varepsilon$
  and decreasing $\delta$ the number of solved subproblems grows. This
  is due to the fact that $\ols{C}$ depends on the approximate
  recession cone $K$. As $\delta$ decreases the rays generating $K$
  will be closer to each other with respect to the truncated Hausdorff
  distance. Therefore, the set $\ols{C}$ will have a larger area and
  it takes more iterations to compute an $\varepsilon$-approximation
  of it. Note, that for $\ed$ equal to $(0.3,0.2)$, $(0.5,0.15)$ or
  $(0.5,0.2)$ the same number of subproblems are solved and the
  approximations have the same number of vertices. For the tolerances
  $(0.3,0.2)$ and $(0.5,0.2)$ the values are identical, because during
  the approximation of $\ols{C}$ the approximation error in Algorithm
  \ref{ACS} changes from a value larger than 0.5 to a value smaller
  than 0.3 in one iteration. Therefore, the resulting
  $\ed$-approximations are identical. For $\ed = (0.5,0.15)$ the
  approximation is different and it is a coincidence that the values
  coincide.
  
\end{example}

\begin{example}\label{exmpl_psd_cone}
  Algorithm 2 can also be used to compute polyhedral approximations of
  closed and pointed convex cones. Consider for example the positive
  semidefinite cone of $2 \times 2$ matrices
  \[
    S = \left\lbrace x \in \R^3 \middle\vert \begin{pmatrix} x_1 & x_3 \\ x_3 &
        x_2 \end{pmatrix} \mgeq 0 \right\rbrace.
  \]
  It is a closed and pointed convex cone with nonempty interior. Thus,
  we can apply Algorithm \ref{AEDA} to it. Since $S$ is a cone, its only
  vertex is the origin and we can terminate the algorithm after $K$ has
  been computed in line 5. Then $K$ is a polyhedral cone and it holds
  $\thaus{K}{S} \leq \delta$. Figure \ref{fig_psd_cone} shows a
  polyhedral approximation of $S$ with 20 extreme rays and
  $\thaus{K}{S} \leq 0.1$.

  \begin{figure}
    \centering
    \includegraphics[scale=0.6]{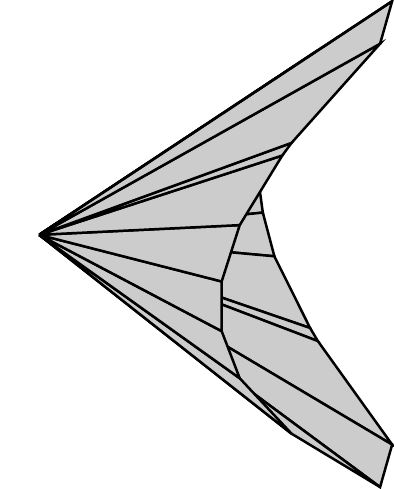}
    \decoRule
    \caption{\label{fig_psd_cone} A polyhedral approximation of the
      cone of $2 \times 2$ positive semidefinite matrices obtained by
      Algorithm \ref{AEDA} for $\delta= 0.1$, see Example
      \ref{exmpl_psd_cone}.}
  \end{figure}
  
\end{example}

\section{Conclusion}
We have introduced the notion of $\ed$-approximations for the
polyhedral approximation of unbounded convex sets. Since polyhedral
approximation in the Hausdorff distance can only be achieved for
unbounded sets under restrictive assumptions, $\ed$-approximations are
of particular interest, because they allow treatment of a larger class
of sets. An important observation is that the recession cones of the
involved sets must play a crucial role in a meaningful concept of
approximation for unbounded sets. We have shown that
$\ed$-approximations define a suitable notion of approximation in the
sense that a sequence of such approximations convergences and that
$\ed$-approximations generalize the polyhedral approximation of
compact sets with respect to the Hausdorff distance. Finally, we have
presented an algorithm that allows for the computation of
$\ed$-approximations of spectrahedra and have shown that the algorithm
is finite.

\section*{Data Availability Statement}
Data sharing is not applicable to this article as no datasets were
generated or analysed during the current study.

\section*{Acknowledgement}
The author thanks the two anonymous reviewers for their insightful
comments that enhanced the quality of this paper.
 
%references
 \bibliographystyle{bib/tfs}
 \bibliography{bib/references}

\begin{thebibliography}{39}
\providecommand{\natexlab}[1]{#1}
\providecommand{\url}[1]{\texttt{#1}}
\providecommand{\urlprefix}{URL }

\bibitem[{Boyd and Vandenberghe(2004)}]{BV04}
S.~Boyd and L.~Vandenberghe.
\newblock \emph{Convex Optimization}.
\newblock Cambridge University Press, Cambridge, 2004.

\bibitem[{Bronshte\u{\i}n(2007)}]{Bro07}
E.~M. Bronshte\u{\i}n.
\newblock \emph{Approximation of convex sets by polyhedra}.
\newblock Sovrem. Mat. Fundam. Napravl. 22 (2007), pp. 5--37.

\bibitem[{Bronshte\u{\i}n and Ivanov(1975)}]{BI75}
E.~M. Bronshte\u{\i}n and L.~D. Ivanov.
\newblock \emph{The approximation of convex sets by polyhedra}.
\newblock Sibirsk. Mat. \v{Z}. 16 (1975), pp. 1110--1112, 1132.

\bibitem[{Cheney and Goldstein(1959)}]{CG59}
E.~W. Cheney and A.~A. Goldstein.
\newblock \emph{Newton's method for convex programming and {T}chebycheff
  approximation}.
\newblock Numer. Math. 1 (1959), pp. 254--268.

\bibitem[{Ciripoi(2019)}]{Cir19}
D.~Ciripoi.
\newblock \emph{Approximation of Spectrahedral Shadows and Spectrahedral
  Calculus}.
\newblock Ph.D. thesis, Friedrich Schiller University Jena, 2019.

\bibitem[{Ciripoi et~al.(2018)Ciripoi, L\"{o}hne, and Wei{\ss}ing}]{CLW18}
D.~Ciripoi, A.~L\"{o}hne, and B.~Wei{\ss}ing.
\newblock \emph{A vector linear programming approach for certain global
  optimization problems}.
\newblock J. Global Optim. 72 (2018), pp. 347--372.

\bibitem[{Duran and Grossmann(1986)}]{DG86}
M.~A. Duran and I.~E. Grossmann.
\newblock \emph{An outer-approximation algorithm for a class of mixed-integer
  nonlinear programs}.
\newblock Math. Programming 36 (1986), pp. 307--339.

\bibitem[{Dörfler et~al.(2021)Dörfler, Löhne, Schneider, and
  Weißing}]{DLSW21}
D.~Dörfler, A.~Löhne, C.~Schneider, and B.~Weißing.
\newblock \emph{A benson-type algorithm for bounded convex vector optimization
  problems with vertex selection}.
\newblock Optimization Methods and Software 0 (2021), pp. 1--21.

\bibitem[{Eaton et~al.(2021)Eaton, Bateman, Hauberg, and Wehbring}]{EBHW21}
J.~W. Eaton, D.~Bateman, S.~Hauberg, and R.~Wehbring.
\newblock \emph{{GNU Octave} version 6.3.0 manual: a high-level interactive
  language for numerical computations}, 2021.
\newblock \urlprefix\url{https://www.gnu.org/software/octave/doc/v6.3.0/}.

\bibitem[{Ehrgott et~al.(2011)Ehrgott, Shao, and Sch\"{o}bel}]{ESS11}
M.~Ehrgott, L.~Shao, and A.~Sch\"{o}bel.
\newblock \emph{An approximation algorithm for convex multi-objective
  programming problems}.
\newblock J. Global Optim. 50 (2011), pp. 397--416.

\bibitem[{Fejes~T\'{o}th(1948)}]{FT48}
L.~Fejes~T\'{o}th.
\newblock \emph{Approximation by polygons and polyhedra}.
\newblock Bull. Amer. Math. Soc. 54 (1948), pp. 431--438.

\bibitem[{Goberna et~al.(2013)Goberna, Iusem, Mart\'{\i}nez-Legaz, and
  Todorov}]{GIMT13}
M.~A. Goberna, A.~Iusem, J.~E. Mart\'{\i}nez-Legaz, and M.~I. Todorov.
\newblock \emph{Motzkin decomposition of closed convex sets via truncation}.
\newblock J. Math. Anal. Appl. 400 (2013), pp. 35--47.

\bibitem[{Goldman and Ramana(1995)}]{GR95}
A.~J. Goldman and M.~Ramana.
\newblock \emph{Some geometric results in semidefinite programming}.
\newblock J. Global Optim. 7 (1995), pp. 33--50.

\bibitem[{Gurari\u{\i}(1965)}]{Gur65}
V.~I. Gurari\u{\i}.
\newblock \emph{Openings and inclinations of subspaces of a {B}anach space}.
\newblock Teor. Funkci\u{\i} Funkcional. Anal. i Prilo\v{z}en. Vyp. 1 (1965),
  pp. 194--204.

\bibitem[{Holmes(1975)}]{Hol75}
R.~B. Holmes.
\newblock \emph{Geometric Functional Analysis and its Applications}.
\newblock Springer-Verlag, New York-Heidelberg, 1975.
\newblock Graduate Texts in Mathematics, No. 24.

\bibitem[{Iusem and Seeger(2010)}]{IS10}
A.~Iusem and A.~Seeger.
\newblock \emph{Distances between closed convex cones: old and new results}.
\newblock J. Convex Anal. 17 (2010), pp. 1033--1055.

\bibitem[{Kamenev(1992)}]{Kam92}
G.~K. Kamenev.
\newblock \emph{A class of adaptive algorithms for the approximation of convex
  bodies by polyhedra}.
\newblock Zh. Vychisl. Mat. i Mat. Fiz. 32 (1992), pp. 136--152.

\bibitem[{Kamenev(1993)}]{Kam93}
G.~K. Kamenev.
\newblock \emph{Efficiency of {H}ausdorff algorithms for polyhedral
  approximation of convex bodies}.
\newblock Zh. Vychisl. Mat. i Mat. Fiz. 33 (1993), pp. 796--805.

\bibitem[{Kamenev(1994)}]{Kam94}
G.~K. Kamenev.
\newblock \emph{Analysis of an algorithm for approximating convex bodies}.
\newblock Zh. Vychisl. Mat. i Mat. Fiz. 34 (1994), pp. 608--616.

\bibitem[{Kelley(1960)}]{Kel60}
J.~E. Kelley, Jr.
\newblock \emph{The cutting-plane method for solving convex programs}.
\newblock J. Soc. Indust. Appl. Math. 8 (1960), pp. 703--712.

\bibitem[{Khachiyan et~al.(2008)Khachiyan, Boros, Borys, Elbassioni, and
  Gurvich}]{KBBEG08}
L.~Khachiyan, E.~Boros, K.~Borys, K.~Elbassioni, and V.~Gurvich.
\newblock \emph{Generating all vertices of a polyhedron is hard}.
\newblock Discrete Comput. Geom. 39 (2008), pp. 174--190.

\bibitem[{Kronqvist et~al.(2016)Kronqvist, Lundell, and Westerlund}]{KLW16}
J.~Kronqvist, A.~Lundell, and T.~Westerlund.
\newblock \emph{The extended supporting hyperplane algorithm for convex
  mixed-integer nonlinear programming}.
\newblock J. Global Optim. 64 (2016), pp. 249--272.

\bibitem[{L\"{o}hne(2011)}]{Loe11}
A.~L\"{o}hne.
\newblock \emph{Vector Optimization with Infimum and Supremum}.
\newblock Vector optimization. Springer, Heidelberg, 2011.

\bibitem[{L\"{o}hne et~al.(2014)L\"{o}hne, Rudloff, and Ulus}]{LRU14}
A.~L\"{o}hne, B.~Rudloff, and F.~Ulus.
\newblock \emph{Primal and dual approximation algorithms for convex vector
  optimization problems}.
\newblock J. Global Optim. 60 (2014), pp. 713--736.

\bibitem[{L\"{o}hne and Wei{\ss}ing(2016)}]{LW16}
A.~L\"{o}hne and B.~Wei{\ss}ing.
\newblock \emph{Equivalence between polyhedral projection, multiple objective
  linear programming and vector linear programming}.
\newblock Math. Methods Oper. Res. 84 (2016), pp. 411--426.

\bibitem[{Lubin et~al.(2018)Lubin, Yamangil, Bent, and Vielma}]{LYBV18}
M.~Lubin, E.~Yamangil, R.~Bent, and J.~P. Vielma.
\newblock \emph{Polyhedral approximation in mixed-integer convex optimization}.
\newblock Math. Program. 172 (2018), pp. 139--168.

\bibitem[{Minkowski(1903)}]{Min03}
H.~Minkowski.
\newblock \emph{{V}olumen und {O}berfläche}.
\newblock Math. Ann. 57 (1903), pp. 447--495.

\bibitem[{Netzer(2011)}]{Net11}
T.~Netzer.
\newblock \emph{Spectrahedra and their Shadows}.
\newblock Habilitation thesis, University of Leipzig, 2011.

\bibitem[{Ney and Robinson(1995)}]{NR95}
P.~E. Ney and S.~M. Robinson.
\newblock \emph{Polyhedral approximation of convex sets with an application to
  large deviation probability theory}.
\newblock J. Convex Anal. 2 (1995), pp. 229--240.

\bibitem[{Rockafellar(1997)}]{Roc70}
R.~T. Rockafellar.
\newblock \emph{Convex Analysis}.
\newblock Princeton Landmarks in Mathematics. Princeton University Press,
  Princeton, NJ, 1997.
\newblock Reprint of the 1970 original, Princeton Paperbacks.

\bibitem[{Rockafellar and Wets(1998)}]{RW98}
R.~T. Rockafellar and R.~J.-B. Wets.
\newblock \emph{Variational Analysis}, vol. 317 of \emph{Grundlehren der
  Mathematischen Wissenschaften [Fundamental Principles of Mathematical
  Sciences]}.
\newblock Springer-Verlag, Berlin, 1998.

\bibitem[{Ruzika and Wiecek(2005)}]{RW05}
S.~Ruzika and M.~M. Wiecek.
\newblock \emph{Approximation methods in multiobjective programming}.
\newblock J. Optim. Theory Appl. 126 (2005), pp. 473--501.

\bibitem[{Schneider(1981)}]{Sch81}
R.~Schneider.
\newblock \emph{Zur optimalen {A}pproximation konvexer {H}yperflächen durch
  {P}olyeder}.
\newblock Math. Ann. 256 (1981), pp. 289--301.

\bibitem[{Toh et~al.(1999)Toh, Todd, and T\"{u}t\"{u}nc\"{u}}]{TTT99}
K.~C. Toh, M.~J. Todd, and R.~H. T\"{u}t\"{u}nc\"{u}.
\newblock \emph{S{DPT}3---a {MATLAB} software package for semidefinite
  programming, version 1.3}.
\newblock Optim. Methods Softw. 11/12 (1999), pp. 545--581.

\bibitem[{T\"{u}t\"{u}nc\"{u} et~al.(2003)T\"{u}t\"{u}nc\"{u}, Toh, and
  Todd}]{TTT03}
R.~H. T\"{u}t\"{u}nc\"{u}, K.~C. Toh, and M.~J. Todd.
\newblock \emph{Solving semidefinite-quadratic-linear programs using {SDPT}3}.
\newblock Math. Program. 95 (2003), pp. 189--217.

\bibitem[{Ulus(2018)}]{Ulu18}
F.~Ulus.
\newblock \emph{Tractability of convex vector optimization problems in the
  sense of polyhedral approximations}.
\newblock J. Global Optim. 72 (2018), pp. 731--742.

\bibitem[{Varadhan(1984)}]{Var84}
S.~R.~S. Varadhan.
\newblock \emph{Large deviations and applications}, vol.~46 of \emph{CBMS-NSF
  Regional Conference Series in Applied Mathematics}.
\newblock Society for Industrial and Applied Mathematics (SIAM), Philadelphia,
  PA, 1984.

\bibitem[{Veinott(1967)}]{Vei67}
A.~F. Veinott, Jr.
\newblock \emph{The supporting hyperplane method for unimodal programming}.
\newblock Operations Res. 15 (1967), pp. 147--152.

\bibitem[{Westerlund and Pettersson(1995)}]{WP95}
T.~Westerlund and F.~Pettersson.
\newblock \emph{An extended cutting plane method for solving convex minlp
  problems}.
\newblock Computers \& Chemical Engineering 19 (1995), pp. 131--136.

\end{thebibliography}

\end{document}